\documentclass{amsart}
\usepackage{url}

\usepackage{datetime}
\usepackage{amsmath,amsfonts,amsthm,amssymb}

\newtheorem{theorem}{Theorem}
\newtheorem{proposition}[theorem]{Proposition}

\newtheorem{lemma}[theorem]{Lemma}

\newtheorem{remark}[theorem]{Remark}
\numberwithin{equation}{section} \numberwithin{theorem}{section}
\newcommand{\R}{\mathbb{R}}
\newcommand{\N}{\mathbb{N}}

\newcommand{\C}{\mathbb{C}}

\renewcommand{\l}{\lambda}
\newcommand{\z}{\zeta}

\newcommand{\mint}{\mathop{\int\hspace{-1.05em}{\--}}\nolimits}

\newcommand{\ep}{\varepsilon} 
\newcommand{\al}{\alpha} 
\newcommand{\Ea}{E_{\al}} 
\newcommand{\ua}{u_{\al}} 
\newcommand{\ula}{u_{\lambda}}

\DeclareMathOperator{\Div}{div}

\begin{document}
\title[Limits of $\al$-harmonic maps]{Limits of $\al$-harmonic maps}

\author{Tobias Lamm}
\address[T.~Lamm]{Institute for Analysis\\ 
Karlsruhe Institute of Technology (KIT)\\ 
Englerstr. 2\\ 76131 Karlsruhe\\ Germany}
\email{tobias.lamm@kit.edu}
\author{Andrea Malchiodi}
\address[A.~Malchiodi]{Scuola Normale Superiore\\ 
Piazza dei Cavalieri 7\\ 
50126 Pisa\\ Italy}
\email{andrea.malchiodi@sns.it}
\author{Mario Micallef}
\address[M.~Micallef]{Mathematics Institute\\ 
University of Warwick\\ 
Coventry CV4 7AL\\ UK}
\email{M.J.Micallef@warwick.ac.uk}

\date{\today}

\subjclass[2000]{}

\begin{abstract} 
Critical points of approximations of the Dirichlet energy \`{a} la Sacks-Uhlenbeck 
are known  to converge to harmonic maps in a suitable sense. 
However, we show that not every harmonic map can be approximated 
by critical points of such perturbed energies. Indeed, we prove that 
constant maps and the rotations of $S^2$ are the only critical points of $\Ea$ 
for maps from $S^2$ to $S^2$ whose $\al$-energy lies below some threshold. 
In particular, nontrivial dilations (which are harmonic) 
cannot arise as strong limits of $\al$-harmonic maps. 
\end{abstract}

\maketitle
\section{Introduction}

Let $(M^2,g)$ and $(N^n,h)$ be smooth, compact 
Riemannian manifolds without boundary and 
let $N$ be isometrically embedded into some $\R^k$. 
(The dimension of $M$ is two and that of $N$ is arbitrary.) 
For every $u\in W^{1,2}(M,N)$ the Dirichlet energy $E(u)$ is defined by 
\begin{equation}
  E(u) = \frac{1}{2}\int_M |\nabla u|^2 \, dA_M = 
  \int_M e(u) \, dA_M,\label{dirichlet}
\end{equation} 
where $e(u)=\frac12 |\nabla u |^2$ is the energy density of $u$.

In a pioneering paper, \cite{sacks81}, 
Sacks and Uhlenbeck introduced, for every $\al>1$ and 
every $u\in W^{1,2\al}(M,N)$, the functional 
$\Ea(u)=\frac{1}{2}\int_M (1+|\nabla u|^2)^\al \, dA_M.$ 
For us, it shall be more convenient to define 
\begin{equation}
  \Ea(u)=\frac{1}{2}\int_M (2+|\nabla u|^2)^\al \, dA_M.\label{energy}
\end{equation}
Critical points of $\Ea$ are called $\al$-harmonic maps and 
they solve the elliptic system
\begin{equation}
  \text{div}\Big((2+|\nabla u|^2)^{\al-1}\nabla u \Big) + 
  (2+|\nabla u|^2)^{\al-1}A(u)(\nabla u,\nabla u)=0,\label{EL}
\end{equation}
where $A$ is the second fundamental form of the embedding 
$N\hookrightarrow \R^k$. 
Critical points of $\Ea$ are smooth (see \cite{sacks81}) and 
therefore we can differentiate the equation \eqref{EL} to get
\begin{equation}
  \Delta u+A(u)(\nabla u,\nabla u) = -2(\al-1) (2+|\nabla u|^2)^{-1} 
  \langle \nabla^2 u, \nabla u \rangle \nabla u. \label{EL2}
\end{equation} 
By a remarkable result of H\'{e}lein, \cite{hel91}, 
critical points of $E$ also turn out to be smooth and satisfy 
$$ 
  \Delta u+A(u)(\nabla u,\nabla u) = 0. 
$$ 
In \cite{sacks81}, Sacks and Uhlenbeck showed that, 
as $\al \downarrow 1$, a sequence of $\al$-harmonic maps 
with uniformly bounded energy converges, away from 
a finite (possibly empty) set of points $p_1, \dotsc ,p_{\ell}$, 
to a harmonic map from $M$ to $N$. Furthermore, 
non-trivial bubbles (harmonic maps from the two-sphere $S^2$) develop 
at each of $p_1, \dotsc ,p_{\ell}$. (This is far from a precise statement 
of the convergence that occurs but it suffices for our purposes.) 
It would be useful to associate a Morse index 
to a harmonic map with bubbles. 
An $\al$-harmonic map has a well-defined Morse index 
(see e.g. \cite{micallef88}, \cite{uhlenbeck72}) and so, 
it seems worthwhile to investigate whether 
every harmonic map from a surface can be captured by 
the Sacks-Uhlenbeck limiting process. 
We shall show that this is not the case, 
even when $M$ and $N$ are the round unit two-sphere $S^2 \subset \R^3$. 

In this case the equation \eqref{EL2} simplifies to
\begin{equation}
  \Delta u+u|\nabla u|^2=-2(\al-1) (2+|\nabla u|^2)^{-1} 
  \langle \nabla^2 u, \nabla u \rangle \nabla u. \label{EL3}
\end{equation} 
For $u \colon S^2 \rightarrow S^2$ we can define the degree of $u$ by
\begin{equation}
  \deg (u)=\frac{1}{4\pi} \int_{S^2} J(u) \, dA_{S^2},\label{deg}
\end{equation}
where 
$$ 
  J(u)=u \cdot e_1(u) \wedge e_2(u)
$$ 
is the Jacobian of $u$, and $(e_1, e_2)$ stands for a local oriented orthonormal frame of $TS^2$. 
For every $u \in W^{1,2\al}(S^2,S^2)$ 
with $\deg(u)=1$ we can estimate 
\begin{align}
  8\pi &= \int_{S^2} (1+J(u)) \, dA_{S^2} \nonumber \\
  &\leqslant \int_{S^2} (1+e(u)) \, dA_{S^2} \label{est1} \\ \nonumber 
  &\leqslant (2^{1-\al} \Ea(u))^{\frac{1}{\al}} 
  (4\pi)^{\frac{\al-1}{\al}}. 
\end{align}
Hence we get 
\begin{equation}
  \Ea(u)\geqslant 2^{2\al+1} \pi \label{est2}
\end{equation}
for every $u$ as above.
On the other hand we have for every $R\in SO(3)$ that 
the map $u^R(x)=Rx$ satisfies 
\begin{equation}
  \Ea(u^R) = 2^{2\al + 1}\pi. \label{est3}
\end{equation}
From \eqref{est1} it follows that 
equality in this estimate is attained only for conformal maps $u$ 
with constant energy density equal to $2$. 
Hence the rotations are the only minimizers of $\Ea$ 
among all maps with degree $1$. 
By contrast we have the following theorem due to Wood and Lemaire 
(see (11.5) in \cite{eelm}). 

\begin{theorem} \label{harm:rat} (\cite{eelm})
The harmonic maps between 2-spheres are precisely 
the rational maps and their complex conjugates 
(i.e., rational in $z$ or $\bar{z}$). 
\end{theorem} 
In particular, a rational map $u$ has energy given by
$E(u) = 4 \pi |\text{deg}(u)|$, 
which is the least energy that a map of this degree can have.
As we shall discuss more fully in a moment, 
the rational maps of degree one include dilations 
which are not minimizers of the $\Ea$ energy for $\al \neq 1$. 

\begin{theorem} \label{t:main} 
There exists $\ep > 0$ and $\overline{\al}-1 > 0$  small 
such that the only critical points $\ua$ of $\Ea$ which satisfy 
$\Ea (\ua) \leqslant 2^{2\al +1} \pi + \ep$ and $\al \leqslant \overline{\al}$ 
are the constant maps and the rotations of the form $u^R(x)=Rx$, $R \in SO(3)$. 
\end{theorem} 

\ 

\begin{remark}\label{r:energy} An upper bound on the energy is necessary 
in order to deduce the conclusions of Theorem \ref{t:main}. 
In Section \ref{s:other} we will construct critical points of $\Ea$ 
of degree one that have large energy and that are not rotations. 
\end{remark}

\

Our proof of Theorem \ref{t:main} goes as follows. After recalling some basic formulas for the M\"obius 
group in Section \ref{s:mob}, we prove in Section \ref{s:clos} that maps with low enough $\Ea$ energy 
must stay close in $W^{1,2}$ to some M\"obius map. We then improve this result in Section \ref{s:closw2p} 
for critical points of $\Ea$ (with low energy),  where we show closeness (after a conformal pull-back) 
to the identity in $W^{2,p}$, where $p>\frac43$ is chosen suitably. 

In Section \ref{s:lbound} we show that elements in the M\"obius group 
that are close to $u$ as in Theorem \ref{t:main} lie in a compact set depending on $\Ea(u_{\al})$. 
The techniques used in this section are similar to those used by Kazdan and Warner and 
also in the study of the semiclassical nonlinear Schr\"odinger equation; 
see for instance, Chapter 8.1 in \cite{amma1}. 
We proceed in section \ref{s:betterclosw2p} to further improve the $W^{2,p}$-closeness, 
and we finally prove our main theorem in Section \ref{s:pf}.

In Section \ref{s:other} we construct a rotationally symmetric $\al$-harmonic map of degree one 
with large energy which is not a rotation. As a byproduct we obtain 
the existence of $\al$-harmonic maps of degree one from the disk to $S^2$ 
which map the boundary circle to a point and 
we also obtain $\al$-harmonic maps of degree one 
which map an annulus to the sphere in such a way that 
the two boundary circles are mapped to antipodal points. 
Note that there are no such harmonic maps. 

\

\noindent {\bf Acknowledgements} T.L. wishes to thank the University of Warwick 
for having hosted him several times during the preparation of this work. 
A.M. has been supported by the PRIN project 
{\em Variational and perturbative aspects of nonlinear differential problems} 
and by the University of Warwick. 
M.M. acknowledges hospitality from the Max-Planck-Institute for Gravitational Physics in Golm 
and the University of Frankfurt.

\section{The Action of the M\"{o}bius Group}\label{s:mob}
Let $\varphi \colon S^2 \to S^2$ be a holomorphic map of degree 1. 
Given an arbitrary map $u \colon S^2 \to S^2$, 
we shall be interested in how $e(u \circ \varphi)$ and 
$\Ea(u \circ \varphi)$ depend on $\varphi$. 
For this, it is convenient to identify 
$S^2 \subset \R^3 $ with $\widehat{\C}=\C \cup \{\infty\}$ 
via the stereographic projection from the north pole. 
If we denote the domain $S^2 \subset \R^3 $ as 
$\{(x,y,z) \in \R^3 : x^2 + y^2 + z^2 = 1\}$ 
and the target $S^2 \subset \R^3 $ as 
$\{(u^1,u^2,u^3) \in \R^3 : (u^1)^2 + (u^2)^2 + (u^3)^2 = 1\}$, 
then the stereographic identifications with $\widehat{\C}$ 
are given by 
$$ 
  x + i y = \frac{2 \z}{1 + |\z|^2}, 
   \quad z = \frac{|\z|^2 - 1}{|\z|^2 + 1}; \qquad\quad 
  u^1 + i u^2 = \frac{2 \eta}{1 + |\eta|^2}, 
   \quad u^3 = \frac{|\eta|^2 - 1}{|\eta|^2 + 1}. 
$$
The inverse maps are 
$$ 
  \z = \frac{x+iy}{1-z}; \qquad\quad 
  \eta = \frac{u^1 + iu^2}{1-u^3}. 
$$ 
\subsection{The M\"{o}bius Group} 

The holomorphic maps of degree one from $\widehat{\C}$ to itself are 
the so-called fractional linear transformations which are of the form 
$$ 
  \z \mapsto \frac{a \z + b}{c \z + d}, \quad ad-bc = 1. 
$$ 
They form a group, called the M\"{o}bius group, 
which is the projective special linear group $PSL(2,\C)$. 
Given $M \in SL(2,\C)$, let $\l, \l^{-1}, \ \l > 0$, be 
the eigenvalues of $MM^*$. The singular value decomposition of matrices (see, e.g., \cite{strang93})
tells us that there exists $U,V \in SU(2)$ such that, 
\begin{equation} \label{eq:SVD} 
  M=UDV^*, \text{ where } D = \begin{pmatrix} \l^{1/2}&0 \\ 0&\l^{-(1/2)} \end{pmatrix} . 
\end{equation} 

Elements of the subgroup $SU(2)$ of $SL(2,\C)$ represent a rotation; 
indeed, if $I$ denotes the $2 \times 2$ identity matrix then, 
$SO(3)$ may be identified with $SU(2)/\{I, -I\}$, 
which establishes $SU(2)$ as the double cover of $SO(3)$. 
The diagonal matrices of the form 
$\begin{pmatrix} \l^{1/2}&0 \\ 0&\l^{-(1/2)} \end{pmatrix}$ 
represent the dilations $m_{\l}$ which are defined by 
$$ 
  m_\l(\z) := \l\z. 
$$ 

\subsection{Energy density in stereographic coordinates} 
A map $u \colon S^2 \to S^2$ shall also be denoted by 
$\eta \colon \widehat{\C} \to \widehat{\C}$. 
However, we shall still denote by $u$ the map to $S^2$ that arises from 
identifying the domain $S^2$ with $\widehat{\C}$. We have: 
\begin{itemize} 
\item the energy density of $u$, $e(u)$, is given by: 
$$ 
  e(u)(\z) = \frac{(1+|\z|^2)^2}{2(1+|\eta|^2)^2}\,|\nabla_0 \eta|^2 
$$ 
where $\nabla_0 \eta$ is the Euclidean gradient of $\eta$ as a map from 
$\C$ to $\C$ with the flat metrics on both domain and target. 
\item The area element $dA_{S^2}$ on the domain $S^2$ is given by: 
$$ 
  dA_{S^2} = \frac{4}{(1+|\z|^2)^2} \, dA_0
$$ 
where $dA_0 := \frac{\sqrt{-1}}{2} d\z \wedge d\bar{\z}$ 
is the Euclidean area element on $\C$. 
\end{itemize} 

\subsection{Transformation of energy density and $\al$-energy 
under composition by a M\"{o}bius transformation} 
Given $M \in SL(2,\C)$ and a map $u \colon \widehat{\C} \to S^2$, 
let $u_M$ be the map defined by 
$$ 
  u_M(\z) = u(M \z) \text{ where, if } 
  M = \begin{pmatrix} a&b \\ c&d \end{pmatrix} \text{ then, by 
  $M \z$ we mean }\frac{a \z+b}{c \z+d}. 
$$ 
We have 
\begin{equation} \label{eq:euM} 
 \begin{split} 
 e(u_M)(\z) &= 
  \frac{(1 + |\z|^2)^2}{2(1 + |\eta(M \z)|^2)^2} 
  \left| \frac{d}{d \z}\left(\frac{a \z+b}{c \z+d}
  \right) \right|^2 |\nabla_0 \eta|^2(M \z) \\ 
  &= 
 \frac{(1 + |\z|^2)^2}{|c \z + d|^4(1 + |M \z|^2)^2} \, 
   \big(e(u)(M \z)\big). 
 \end{split} 
\end{equation} 
Now 
\begin{equation} \label{eq:den} 
 \begin{split} 
 |c \z + d|^2(1 + |M \z|^2) &= |a \z + b|^2 + |c \z + d|^2 \\ 
  &= \left| \begin{pmatrix} a&b \\ c&d \end{pmatrix} 
   \begin{pmatrix}\z \\ 1 \end{pmatrix} \right|^2 \\ 
  &= \left| \begin{pmatrix} \l^{1/2}&0 \\ 0&\l^{-(1/2)} \end{pmatrix} 
   \begin{pmatrix}\z \\ 1 \end{pmatrix} \right|^2 
    \quad \text{(by \eqref{eq:SVD})} \\ 
  &= \frac{\l^2 |\z|^2 + 1}{\l}. 
 \end{split} 
\end{equation} 
Using \eqref{eq:den} in \eqref{eq:euM} gives 
\begin{equation} \label{euM_eu} 
e(u_M)(\z) = \frac{\l^2 (1 + |\z|^2)^2}{(1 + \l^2 |\z|^2)^2} \, 
   \big(e(u)(M \z)\big).
\end{equation} 

The transformation relation \eqref{euM_eu} allows us to 
restrict our attention to the dilations $m_{\l}$. Set 
$\ula = u \circ m_{\l}$, i.e., $\ula(\z) = u(\l \z)$ and set 
\begin{equation}\label{eq:chil}
   \chi_\l(\z) = \frac{(1+\l^2 |\z|^2)^2}{\l^2(1 + |\z|^2)^2}. 
\end{equation}
Then 
$$ 
  e(u)(\l \z) = \chi_\l(\z)  \big(e(\ula)(\z)\big) 
$$ 
for every $\l >0$ and therefore, 
\begin{align*} 
  \Ea(u)&=2^{\al - 1} \int_{\C}\big(1+e(u)(\z)\big)^{\al} 
    \frac{4}{(1 + |\z|^2)^2} \, dA_0(\z) \\ 
  &=2^{\al - 1} \int_{\C}\big(1+e(u)(\l\z)\big)^{\al} 
    \frac{4 \l^2}{(1 + |\l\z|^2)^2} \, dA_0(\z) \\ 
  &=2^{\al - 1} \int_{\C}\big(1+\chi_\l(\z)e(\ula)(\z)\big)^{\al} 
    \frac{4}{\chi_\l(\z)(1 + |\z|^2)^2} \, dA_0(\z), 
\end{align*} 
that is, 
\begin{equation} \label{Ea_eallam} 
\Ea(u) = E_{\al,\l}(\ula) = E_{\al, \l^{-1}} (u_{\l^{-1}})
\end{equation} 
where $E_{\al,\l}$ is the functional defined by 
\begin{equation} \label{def:Ealphlam} 
  E_{\al,\l}(v) = \frac{1}{2} \int_{S^2} \left( 2 + \chi_\l 
  |\nabla_{S^2} v|^2 \right)^{\al} \frac{1}{\chi_\lambda} dA_{S^2}. 
\end{equation} 
Clearly $u$ is a critical point of $\Ea$ if, and only if, 
$\ula$ is a critical point of $E_{\al,\l}$. 
Moreover, due to the above symmetry of $\Ea$ in $\l$, $\l^{-1}$, 
we assume throughout the rest of the paper that $\l \geqslant 1$.

\begin{proposition} \label{p:ELal} 
If $\chi_\l$ is as in \eqref{eq:chil}, 
the Euler Lagrange equation satisfied by a critical point $v$ 
of $E_{\al,\l}$ is 
$$
  \Delta v + |\nabla v|^2 v + f_1 + f_2 = 0, 
$$
where 
\begin{align}
  f_1&:=(\al - 1) \left( 
  \frac{\chi_\l \nabla (|\nabla v|^2) \cdot \nabla v}{2+\chi_\l |\nabla v|^2} 
   \right) \label{def.fi} \\ 
\noalign{and} 
  f_2&:=(\al - 1) \left( 
 \frac{\chi_\l |\nabla v|^2 \nabla \log \chi_\l \cdot \nabla v}{2+\chi_\l |\nabla v|^2} 
  \right). \label{def.fi1} 
\end{align} 
\end{proposition} 
The proof of this proposition is just a straightforward computation. 

\section{Closeness to the M\"obius group} \label{s:clos}

\noindent The aim of this section is to prove the following proposition. 

\begin{proposition}\label{p:clomob} 
There exists $\delta^* > 0$ such that, 
for any $\delta \in (0, \delta^*)$ there exists $\ep > 0$ such that, 
if $1 \leqslant \al \leqslant 2$ and if $\Ea(u) \leqslant 2^{2\al+1} \pi + \ep$, 
where $u$ is of degree $1$, then there exists $M \in PSL(2,\C)$ such that 
\begin{align}
   \left\|  \nabla (u_M - Id) \right\|_{L^2(S^2)} \leqslant \delta. \label{eq:delta-close}
\end{align}
Furthermore, there is a fixed constant $C$ such that, 
if $\l \geqslant 1$ is the largest eigenvalue of $MM^*$ 
(see \eqref{eq:SVD}) then 
\begin{equation} \label{eq:alphalambd} 
(\al - 1) (\log \l) \min\{\log \l, 1\} \leqslant C \delta. 
\end{equation} 
\end{proposition}

\noindent The proof of the above proposition relies on 
the three lemmas below. 

\begin{lemma} \label{l:closeMob}
Given $\delta > 0$, there exists $\ep > 0$, sufficiently small, 
with the following property: 
for all $\al \geqslant 1$, if $u \in W^{1,2\al} (S^2, S^2)$ is of degree 1 
and $\Ea(u) \leqslant 2^{2\al+1} \pi + \ep$, 
there exists $M \in PSL(2,\C)$ such that 
\begin{equation} \label{L2close} 
   \left\| \nabla (u_M - Id)  \right\|_{L^2(S^2)} \leqslant \delta. 
\end{equation} 
\end{lemma}

\begin{proof} If $\Ea(u) \leqslant 2^{2\al+1} \pi + \ep$ then by \eqref{est1} 
we have 
\begin{align*} 
E_1(u) &= \int_{S^2} (1+e(u)) \, dA_{S^2} \\ 
 &\leqslant \left(\frac{2^{1-\al}\Ea(u)}{4 \pi}\right)^{\frac{1}{\al}} 4 \pi \\ 
 &\leqslant \left(1+\frac{\ep}{2^{2 \al + 1} \pi}\right)^{\frac{1}{\al}} 
   8 \pi \\ 
 &\leqslant 8 \pi + \ep. 
\end{align*} 
If, for a contradiction, the lemma were not true, 
we could find a sequence $\ep_n \downarrow 0$, 
a sequence $u_n\in W^{1,2}(S^2,S^2)$ of degree one, with $E_1(u_n) \leqslant 8 \pi + \ep_n$ 
and $\delta > 0$ such that 
\begin{equation} \label{eq:delta}
  \left\| \nabla \big((u_n)_M - Id \big)  \right\|_{L^2(S^2)} > \delta 
   \quad \hbox{ for all } M \in PSL(2,\C). 
\end{equation}
But $u_n$ would then be a minimising sequence for $E_1$ of degree one
and therefore, by Theorem 1 in \cite{duzkuw}, 
there exists $M_n \in PSL(2,\C)$ such that $(u_n)_{M_n}$ 
converges strongly in Dirichlet norm 
to a degree one minimiser $u_{\infty}$ of $E_1$. 
(We remark that, by energetic reasons, 
multiple splitting into maps of different degrees is excluded.) 
By Theorem \ref{harm:rat}, $u_{\infty}$ is of the form 
$\z \mapsto M_{\infty} \z$ for some $M_{\infty} \in PSL(2,\C)$. 
By the conformal invariance of the Dirichlet integral we have that 
$$
 \left\| \nabla 
   \big((u_n)_{M_nM_{\infty}^{-1}} - Id \big) 
 \right\|_{L^2(S^2)} 
\to 0. 
$$
This then contradicts \eqref{eq:delta} and concludes the proof. 
\end{proof}

We still need to establish a bound on 
the largest eigenvalue $\l$ of $MM^*$ in the previous lemma. 
The rough plan for doing this is that, because of 
the closeness in Dirichlet norm provided by \eqref{L2close}, 
$E_{\al,\l}(u_M)$ should be close to $E_{\al,\l}(Id)$. 
We should then be able to explicitely describe how 
$E_{\al,\l}(Id)$ grows with $\l$. Recall that the relation between 
$\Ea$ and $E_{\al,\l}$ is given by \eqref{def:Ealphlam}.
This plan is executed in the next two lemmas.   

\begin{lemma} \label{l:Eaclose} 
If $\lambda \geqslant 1$ and $1 \leqslant \al \leqslant 2$, 
we have
\begin{equation} \label{eq:Eaclose} 
  E_{\al,\l}(v) - E_{\al,\l}(Id) \geqslant 
   - \al 2^{\al - 2} (1 + \l^2)^{\al - 1} 
    \| \, | \nabla_{S^2} v |^2 - 2 \,  \|_{L^1(S^2)}. 
\end{equation} 
\end{lemma} 

\begin{proof} 
By the mean value theorem, there is a positive function 
$g : S^2 \to \R_+$ whose value at $p$ lies between 
$|\nabla_{S^2} v(p)|^2$ and 2 = $|\nabla_{S^2} Id|^2$ such that 
\begin{equation} \label{eq:mvt} 
  E_{\al,\l}(v) - E_{\al,\l}(Id) = 
  \frac{\al}{2} \int_{S^2} \left( 2 + \chi_\l g
   \right)^{\al-1} (|\nabla_{S^2} v|^2 - 2) \, dA_{S^2}. 
\end{equation} 
Let 
$$ 
  A_+ := \{p \in S^2 : |\nabla_{S^2} v(p)|^2 \geqslant 2\} 
  \quad\text{and}\quad 
  A_- := \{p \in S^2 : |\nabla_{S^2} v(p)|^2 < 2\}. 
$$ 
Then, on $A_+$ $g \geqslant 2$ and on $A_-$ $g \leqslant 2$. Therefore, 
$$ 
  \int_{A_+} \left( 2 + \chi_\l g
   \right)^{\al-1} (|\nabla_{S^2} v|^2 - 2) \, dA_{S^2} 
  \geqslant  2^{\al-1} \int_{A_+} \left( 1 + \chi_\l \right)^{\al-1} 
   (|\nabla_{S^2} v|^2 - 2) \, dA_{S^2} 
$$ 
and, since $(|\nabla_{S^2} v|^2 - 2)$ is negative on $A_{-}$, 
$$ 
  \int_{A_-} \left( 2 + \chi_\l g
   \right)^{\al-1} (|\nabla_{S^2} v|^2 - 2) \, dA_{S^2} 
  \geqslant  2^{\al-1} \int_{A_-} \left( 1 + \chi_\l \right)^{\al-1} 
   (|\nabla_{S^2} v|^2 - 2) \, dA_{S^2}. 
$$ 
It follows that 
\begin{equation} \label{eq:mvt_est} 
  \int_{S^2} \left( 2 + \chi_{\l} g \right)^{\al-1} 
   (|\nabla_{S^2} v|^2 - 2) \, dA_{S^2} \geqslant 
   2^{\al-1} \int_{S^2} \left( 1 + \chi_\l \right)^{\al-1} 
   (|\nabla_{S^2} v|^2 - 2) \, dA_{S^2}. 
\end{equation} 
Now $\sup\limits_{S^2}\chi_{\l} = \l^2$ and therefore, 
\begin{equation} \label{eq:mvt_bd} 
 \left|\int_{S^2} \left( 1 + \chi_\l \right)^{\al-1} 
   (|\nabla_{S^2} v|^2 - 2) \, dA_{S^2}\right| 
  \leqslant (1 + \l^2)^{\al - 1} 
    \| \, | \nabla_{S^2} v |^2 - 2 \, \|_{L^1(S^2)}. 
\end{equation} 

Estimate \eqref{eq:Eaclose} is established by putting together 
\eqref{eq:mvt}, \eqref{eq:mvt_est} and \eqref{eq:mvt_bd}. 
\end{proof} 

The next lemma describes how $E_{\al,\l}(Id)$ grows with $\l$. 
\begin{lemma}\label{l:eulambda} 
We have that 
\begin{align}
E_{\al,\l}(Id) = \Ea(m_{\l^{-1}}) = \Ea(m_{\l}). \label{eq:relenergy}
\end{align}
Moreover, by letting 
\begin{equation} \label{def:xi} 
\xi(\al, \l) := \Ea(m_{\l}) - 2^{2 \al + 1} \pi,
\end{equation} 
there exists a fixed constant $C$ such that, 
for $1 < \al \leqslant 2$, 
\begin{equation} \label{eq:lowerbdea}
 \xi(\al,\l) \geqslant 
  \begin{cases} 
    C \l^{2 \al - 2}, & \text{if }(\al - 1)\log \l \geqslant 2, \\ 
    C (\al - 1) \log \l, & \text{if } 
      (\al - 1) \leqslant (\al - 1)\log \l \leqslant 2 \\ 
    C (\al - 1)(\log \l)^2, & \text{if }0 \leqslant \log \l \leqslant 1. 
  \end{cases} 
\end{equation}
Additionally, $\Ea(m_{\l})$ is increasing in $\l$ and 
we have for $0 \leqslant (\al-1) \log \l  \leqslant 2$ that
\begin{equation}\label{growthenergy}
\frac{\partial}{\partial \log \lambda} \Ea(m_\l) = 
\frac{\partial}{\partial \log \l} E_{\al,\l} (Id) \geqslant 
C (\al - 1) \frac{|\log \lambda|}{1+ |\log \lambda|}.  
\end{equation}
\end{lemma}

\begin{proof} 
We start by obtaining an explicit formula for $\Ea(m_{\l})$: 
set $r := |\z|$ and then, as we saw in \S 2, 
$$ 
  e(m_{\l})(\z) =  \l^2 \frac{(1+r^2)^2}{(1 + \l^2 r^2)^2} 
  = \frac{1}{\chi_{\l}(\z)}. 
$$ 
So, 
$$ 
  \Ea(m_{\l}) = 2^{\al - 1} 8 \pi \int_0^{\infty} \left( 
   1 + \frac{\l^2 (1+r^2)^2}{(1 + \l ^2 r^2)^2} 
  \right)^{\al} \frac{r}{(1+r^2)^2} \, dr . 
$$ 
We make the change of variable 
\[ 
w := \l \frac{1+r^2}{1 + \l^2 r^2} 
\] 
for which 
\[ 
dw = 
2 \l r \frac{1-\l^2}{(1 + \l^2 r^2)^2}  dr
\] 
and obtain 
\[ 
\Ea(m_{\l}) = 2^{\al + 1} \pi \frac{\l}{\l^2 - 1} 
\int_{1/\l}^{\l} (1 + w^2)^{\al} w^{-2} \, dw. 
\] 
Setting $\l := e^{\tau}$ and $w := e^t$ yields: 
\begin{align}\label{eq:Ea_mtau} 
\Ea(m_{e^{\tau}}) &= 2^{\al + 1} \pi \frac{e^{\tau}}{e^{2 \tau} - 1} 
  \int_{-\tau}^{\tau} (1 + e^{2t})^{\al} e^{-t} \, dt 
  \nonumber \\[3\jot] 
&= \frac{2^{\al} \pi}{\sinh \tau} 
  \int_{-\tau}^{\tau} (e^{-t} + e^t)^{\al} e^{(\al - 1)t} \, dt 
  \nonumber\\[3\jot] 
&= \frac{2^{2 \al + 1} \pi}{\sinh \tau} 
\int_0^{\tau} (\cosh t)^{\al} \cosh((\al - 1)t) \, dt 
\end{align} 
where we have used 
\[ 
\int_{-\tau}^0 (e^{-t} + e^t)^{\al} e^{(\al - 1)t} \, dt 
= 
\int_0^{\tau} (e^{-t} + e^t)^{\al} e^{-(\al - 1)t} \, dt \, . 
\] 
It is immediate from this expression for $\Ea(m_{\l})$ that $\Ea(m_{\l}) = \Ea(m_{\l^{-1}})$ and 
the relation \eqref{eq:relenergy} then follows by taking \eqref{Ea_eallam} into account. 

As expected we have 
$E_1(m_{e^{\tau}}) = 8 \pi \ \forall \, \tau \in \R$ and 
$\Ea(m_1) = 2^{2 \al + 1} \pi$. 

\bigskip 
\noindent It will be convenient to set 
\[ 
\beta := (\al - 1), 
\] 
to make the change of variables 
\[ 
s := \beta t, \quad \sigma := \beta\tau = (\al - 1)\log\l 
\] 
and to introduce the functions 
\begin{equation}\label{eq:gGdef} 
  \begin{split} 
 g(s) &:= (\cosh(s/\beta))^{\beta}\cosh s \\[\jot] 
 \text{and}& \\ 
 G(\sigma) &:= \frac{1}{\beta \sinh (\sigma/\beta)} 
   \int_0^{\sigma} (\cosh (s/\beta)) g(s) \, ds. 
  \end{split} 
\end{equation} 
Then \eqref{eq:Ea_mtau} becomes 
\begin{equation} \label{eq:EaG} 
\Ea(m_{e^{(\sigma/\beta)}}) = 
\frac{2^{2 \al + 1} \pi}{\beta \sinh (\sigma/\beta)} 
\int_0^{\sigma} (\cosh (s/\beta)) g(s) \, ds = 
2^{2 \al + 1} \pi G(\sigma). 
\end{equation} 
The lower bound $\cosh t > \frac12 e^t$ yields 
\[ 
g(s) > \left(\frac{e^{s/\beta}}{2}\right)^{\beta} \frac{e^s}{2} = \frac{e^{2s}}{2^{\al}}. 
\] 

We shall now prove the first inequality in \eqref{eq:lowerbdea}. 
So, we assume that $\sigma \geqslant 2$ and $1 < \al \leqslant 2$ 
and estimate $G$ from below as follows: 
\begin{align*} 
G(\sigma) &> \frac{1}{\beta \sinh (\sigma/\beta)} 
\int_{\sigma - 1}^{\sigma} (\cosh (s/\beta)) g(s) \, ds \\ 
&> \frac{1}{2^{\al} \beta \sinh (\sigma/\beta)} 
\int_{\sigma - 1}^{\sigma} (\cosh (s/\beta)) e^{2s} \, ds \\ 
&> \frac{e^{(2 \sigma - 2)}}{2^{\al}} \ 
\frac{1}{\beta \sinh (\sigma/\beta)} \int_{\sigma - 1}^{\sigma} (\cosh (s/\beta)) \, ds \\ 
&> \frac{e^{2 \sigma}}{2 e^2} \ \frac{\sinh (\sigma/\beta) - \sinh ((\sigma-1)/\beta)}{\sinh (\sigma/\beta)}. 
\end{align*} 
Keeping in mind that $0 \leqslant \beta \leqslant 1$, we have, 
\[ 
\sinh (\sigma/\beta) - \sinh ((\sigma-1)/\beta) > \frac{e^{\sigma/\beta}}{2} (1 - e^{-1/\beta}) 
> \sinh (\sigma/\beta) \left(\frac{e-1}{e}\right). 
\] 
It follows that 
\[ 
G(\sigma) - 1 > e^{2 \sigma} \left(\frac{e-1}{2e^3} - \frac{1}{e^4}\right), 
\] 
i.e., if $(\al - 1) \log \l \geqslant 2$ and $1 < \al \leqslant 2$ then 
\[ 
\xi(\al,\l) \geqslant 2^{2 \al + 1} \pi \left(\frac{e^2-e-2}{2e^4}\right) \l^{2 \al - 2} 
\] 
as claimed. 

\bigskip 
To estimate $G(\sigma) - 1$ from below for $\sigma \in [0,2]$, 
we calculate $G'(\sigma)$ from \eqref{eq:gGdef}: 
\[ 
G'(\sigma) = 
\frac{\cosh (\sigma/\beta)}{\beta \sinh (\sigma/\beta)} g(\sigma) - 
\frac{\cosh (\sigma/\beta)}{\beta^2 \sinh^2 (\sigma/\beta)} 
\int_0^{\sigma} (\cosh (s/\beta)) g(s) \, ds. 
\] 
Now 
\[ 
\frac{1}{\beta \sinh (\sigma/\beta)} 
\int_0^{\sigma} (\cosh (s/\beta)) g(s) \, ds = 
g(\sigma) - \frac{1}{\sinh (\sigma/\beta)} 
\int_0^{\sigma} (\sinh (s/\beta)) g'(s) \, ds . 
\] 

\bigskip 
Differentiating the expression for $g$ from \eqref{eq:gGdef} gives 
\begin{align*} 
g'(s) &= (\cosh(s/\beta))^{\beta-1} 
  (\sinh(s/\beta)\cosh s + \cosh(s/\beta) \sinh s) \nonumber \\ 
& = (\cosh(s/\beta))^{\beta-1} \sinh(\al s/\beta) . 
\end{align*} 

Therefore, we obtain:  
\begin{equation} \label{eq:Gprime} 
G'(\sigma) = \frac{\cosh (\sigma/\beta)}{\beta \sinh^2 (\sigma/\beta)} 
\int_0^{\sigma} (\sinh (s/\beta)) (\cosh(s/\beta))^{\beta - 1} 
\sinh(\al s/\beta) \, ds . 
\end{equation} 

We shall estimate $G'$ from below differently in the two regimes 
$0 \leqslant \sigma \leqslant \beta$ and $0 < \beta \leqslant \sigma \leqslant 2$. 
We start with the latter case for which we shall show that 
$G'$ is bounded below by a positive constant, independent of $\beta$. 

\bigskip 
Using $\dfrac{\cosh (\sigma/\beta)}{\sinh (\sigma/\beta)} > 1$ and 
$\dfrac{\sinh(\al s/\beta)}{\cosh(s/\beta)} \geqslant \tanh(\al s/\beta)$ in \eqref{eq:Gprime}, 
we obtain, for $\theta \in (0,1)$ and $\beta \leqslant \sigma$, 

\begin{align*} 
G'(\sigma) &> \frac{1}{\sinh (\sigma/\beta)} 
 \int_{\theta \beta}^{\sigma} (\tfrac{1}{\beta} \sinh (s/\beta)) 
  (\cosh(s/\beta))^{\beta} \tanh(\al s/\beta) \, ds \\ 
 & \geqslant \tanh \theta \ 
   \frac{\cosh(\sigma/\beta) - \cosh \theta}{\sinh (\sigma/\beta)}\\
 & \geqslant \tanh \theta \ \left( 1- \frac{\cosh \theta}{\sinh 1} \right) ,
\end{align*} 
where we also used that $\tanh(\al \theta) \geqslant \tanh \theta$ and 
$\cosh(s/\beta) \geqslant 1$ in the second estimate.

We now choose $\theta > 0$ so that $\cosh \theta \leqslant \frac12 \sinh 1$ and 
deduce that there exists $C > 0$, independent of anything, such that 
if $\al > 1$ and $\l \geqslant e$, i.e., $\tau \geqslant 1$ and $0 < \beta \leqslant \sigma$ then 
\begin{align}
G'(\sigma) \geqslant C > 0. \label{eq:EstGprime1}
\end{align} 
It follows that for $0 < \beta \leqslant \sigma$ we get
\begin{equation} \label{eq:Gest_mid} 
G(\sigma) \geqslant G(\beta) + C(\sigma - \beta). 
\end{equation} 

The lower bound on $G'$ for $\sigma \in (0,\beta]$ is straightforward. 
First use the inequality 
$\cosh(\sigma/\beta) (\cosh(s/\beta))^{\beta - 1} \geqslant (\cosh(s/\beta))^{\beta } \geqslant 1$ 
for every $s \in [0,\sigma]$ in \eqref{eq:Gprime} to get 
\[ 
G'(\sigma) \geqslant \frac{1}{\beta \sinh^2 (\sigma/\beta)} 
   \int_0^{\sigma} 
  (\sinh (s/\beta)) \sinh(\al s/\beta) \, ds . 
\] 
Next, use $(\sinh (s/\beta)) \sinh(\al s/\beta) \geqslant \frac{s^2}{\beta^2}$ 
and the inequality $\sinh x \leqslant x (\cosh x)$ for $x \geqslant 0$ to get 
\begin{align} 
G'(\sigma) &\geqslant \frac{1}{\beta (\cosh(\sigma/\beta))^2 \sigma^2} 
        \int_0^{\sigma} s^2 \, ds \nonumber \\ 
 & \geqslant \frac{\sigma}{3 \beta (\cosh 1)^2}; 
   \qquad \text{we have used }0 \leqslant \sigma/\beta \leqslant 1.  \label{eq:EstGprime2}
\end{align}
It follows that, 
\begin{equation} \label{eq:Gest_small} 
\text{for }0 \leqslant \sigma \leqslant \beta, \quad 
G(\sigma) - G(0) \geqslant \frac{\sigma^2}{6 \beta (\cosh 1)^2} 
\geqslant \frac{(\al-1)(\log \l)^2}{6(\cosh 1)^2}. 
\end{equation} 

We can now establish the last two estimates in \eqref{eq:lowerbdea}. 
If $\al - 1 \leqslant (\al - 1) \log \l \leqslant 2$ then, 
by \eqref{eq:Gest_mid} and \eqref{eq:Gest_small} we have that 
\[ 
\xi(\al,\l) \geqslant 2^{2 \al + 1} \pi \left( \big(G(\al - 1) -1 \big) 
  + C (\al - 1)(\log \l -1) \right) \geqslant C (\al -1)\log \l. 
\] 
If $ \log \l \leqslant 1$ then, we obtain again from \eqref{eq:Gest_small} that 
\[ 
\xi(\al,\l) \geqslant \frac{2^{2 \al + 1} \pi}{6 (\cosh 1)^2} 
  (\al - 1) (\log \l)^2 . 
\] 

Finally, $\Ea(m_{\l})$ increases with $\l$ because, from 
\eqref{eq:Gprime}, $G'$ is evidently positive. 
Moreover, in order to show \eqref{growthenergy} 
we note that it follows from \eqref{eq:EaG} that 
\[
\frac{\partial}{\partial \log \l} E_{\al,\l}(Id)= (\al - 1)2^{2 \al + 1} \pi G'((\al - 1)\log \l).
\]
For $1\leqslant \log \l \leqslant 2(\al - 1)^{-1}$ we use \eqref{eq:EstGprime1} in order to get
\[
\frac{\partial}{\partial \log \l} E_{\al,\l}(Id)\geqslant C (\al-1) \geqslant C(\al-1) \frac{|\log \l|}{1+|\log \l|}.
\]
For $0<\log \l \leqslant 1$ we use \eqref{eq:EstGprime2} to conclude
\[
\frac{\partial}{\partial \log \l} E_{\al,\l}(Id)\geqslant C(\al-1) \log \l \geqslant C(\al - 1) \frac{|\log \l|}{1+|\log \l|}. 
\]
The proof of Lemma \ref{l:eulambda} is complete.
\end{proof}

We can now give the 
\begin{proof}[Proof of Proposition \ref{p:clomob}] 
Having proved Lemma \ref{l:closeMob}, 
it only remains to establish \eqref{eq:alphalambd}. 
Apply Lemma \ref{l:Eaclose} with $v = u_M$, 
$M$ as provided by \eqref{L2close} and $\l \geqslant 1$ equal to 
the largest eigenvalue of $MM^*$. Then, 
with $\delta$ as in \eqref{L2close}, we have 
\begin{equation} \label{eq:estl1} 
 2^{2 \al + 1} \pi + \ep \geqslant \Ea(u) = E_{\al,\l}(u_M) 
  \geqslant E_{\al,\l}(Id) - \al \pi 2^{ 2\al +1} \l^{ 2\al - 2} \delta , 
\end{equation} 
where we used that
\begin{align*}
\| |\nabla_{S^2} u_M |^2 - 2 \,  \|_{L^1(S^2)} 
\leqslant& 
\left\| \nabla (u_M - Id) \right\|_{L^2(S^2)} 
\left\| \nabla (u_M + Id) \right\|_{L^2(S^2)}\\ 
\leqslant& 
\delta \sqrt{(8 \pi + \ep)(8\pi)} \leqslant \delta (16 \pi).
\end{align*}
Recall that 
\[ 
E_{\al,\l}(Id) = \Ea(m_{\l}) = 2^{2 \al + 1} \pi + \xi(\al, \l) 
\] 
and observe that $\ep$ in Lemma \ref{l:closeMob} can be chosen 
no larger than $\delta$. Therefore, \eqref{eq:estl1} can be rewritten as 
\begin{equation} \label{eq:estl2} 
\delta (1 + C' \l^{2 \al - 2}) \geqslant \xi(\al, \l). 
\end{equation} 

If $(\al - 1) \log \l \geqslant 2$, 
i.e. $\l^{2 \al - 2} \geqslant e^4$, 
then \eqref{eq:lowerbdea} provides the lower bound 
$\xi(\al, \l) \geqslant C \l^{2 \al - 2}$. 
So, \eqref{eq:estl2} cannot hold if 
$0 \leqslant \delta < \delta^* := \min\{\frac{C}{2C'}, \frac{C}{2}e^4\}$. 
Therefore, $\l^{2 \al - 2}$ must be less than $e^4$ 
and so, from \eqref{eq:lowerbdea} and \eqref{eq:estl2}, 
we deduce that 
\[ 
\delta (1 + C' e^4) \geqslant C (\al - 1) (\log \l) \min\{\log \l, 1\}. 
\] 
\end{proof} 

\section{Closeness in the $W^{2,p}$-norm} \label{s:closw2p}
In this section we prove a refinement of Proposition \ref{p:clomob}, 
showing closeness between $u_M$ and the identity 
in $W^{2,p}, \ p \in (\frac43,\frac32]$. 
The reason for this range of $p$ will become apparent in Proposition \ref{p:diff5}. 

\begin{proposition}\label{p:clomob2}
There exist $1 < \al_0$, $\delta_0>0$  and a constant $C$ depending only on $\al_0$ and $\delta_0$ 
such that, for every $1 < \al \leqslant \al_0$, every $0<\delta \leqslant \delta_0$ and 
every critical point $v\in W^{1,2\al}(S^2,S^2)$ of $E_{\al,\l}$ 
satisfying \eqref{eq:delta-close} and \eqref{eq:alphalambd} we have, for any $p \in (\frac43,\frac32]$, 
\begin{equation}\label{eq:sup2p} 
 \| v-Id\|_{L^\infty(S^2)} + 
 \left\|  \nabla (v - Id) \right\|_{W^{1,p}(S^2)} 
 \leqslant C (\delta+\al - 1). 
\end{equation} 
\end{proposition}
\begin{proof}
We define a map $\psi \colon S^2\to \R^3$ by
\[
v=Id+\psi
\]
and we obtain from Proposition \ref{p:clomob} that
\[
\| \nabla \psi \|_{L^2(S^2)} \leqslant \delta.
\]
By Proposition \ref{p:ELal}, $\psi$ satisfies 
\begin{align}
 \Delta \psi=& -2 \psi-2 \langle  \nabla \psi, \nabla Id \rangle Id - 
 |\nabla \psi|^2 \psi-2\langle\nabla \psi ,\nabla Id \rangle \psi - 
 |\nabla \psi|^2 Id - f_1 - f_2. \label{eqnpsi}
\end{align} 
We shall first estimate the average of $\psi$ 
by integrating this equation and 
observing from \eqref{def.fi} that 
\begin{equation}\label{eq:f1bd} 
|f_1(\z)| \leqslant C(\al - 1) |\nabla^2 v (\z)|\leqslant C(\al -1) (1+|\nabla^2 \psi(\z)|)
\end{equation} 
and that 
\begin{equation}\label{eq:f2bd} 
|f_2(\z)| \leqslant C(\al - 1) |(\nabla \log \chi_{\l})(\z)| \, 
|\nabla v(\z)| . 
\end{equation} 
When integrating \eqref{eqnpsi}, keep also in mind that 
$\|\psi\|_{L^\infty(S^2)} \leqslant 2$ and make use of
Proposition \ref{p:clomob} and Lemma \ref{l:intchi} 
to conclude that 
\begin{align}
|\mint_{S^2} \psi \, dA_{S^2} |\leqslant 
 & C\delta + C(\al -1) \|\nabla^2 v \|_{L^1(S^2)} + 
   C(\al - 1) \| \nabla v \|_{L^2(S^2)}
      \| \nabla \log \chi_\l \|_{L^2(S^2)} \nonumber \\
\leqslant & C (\delta + \al - 1) + C(\al -1) \|\nabla^2 \psi \|_{L^1(S^2)}. 
\label{estmeanpsi}
\end{align} 
 
This estimate on the average of $\psi$ allows us to 
use standard $L^p$-estimates for the Laplacian 
and the Sobolev-Poincar\'{e} inequality to conclude that, 
for every $p \in (\frac43,\frac32]$,
\begin{align*}
\| \nabla \psi \|_{W^{1,p}(S^2)} \leqslant 
 & C\left(\|\Delta \psi \|_{L^p(S^2)} + 
   \|\psi \|_{L^p(S^2)}\right)\\
\leqslant & C\left(\|\Delta \psi \|_{L^p(S^2)} + 
   \|\nabla \psi \|_{L^2(S^2)} + |\mint_{S^2} \psi \, dA_{S^2} | \right)\\
\leqslant & C\left( \|\Delta \psi \|_{L^p(S^2)} + 
   \delta + \al - 1 + (\al -1) \|\nabla^2 \psi \|_{L^p(S^2)} \right). 
\end{align*} 
By picking $\al_0 > 1$ sufficiently close to 1 so that $C (\al_0-1)\leqslant \frac12$ we get
\begin{equation}
\| \nabla \psi \|_{W^{1,p}(S^2)} \leqslant C \left( 
  \|\Delta \psi \|_{L^p(S^2)} + \delta + \al - 1 \right). \label{est:lp} 
\end{equation} 
The plan now is to estimate $\|\Delta \psi \|_{L^p(S^2)}$, 
by using \eqref{eqnpsi}. The $L^p$ norm of 
the right hand side of \eqref{eqnpsi} requires us to estimate 
the $L^{2p}$-norm of $\nabla \psi$ which we do by means of 
the Gagliardo-Nirenberg interpolation inequality: 
\[ 
\| \nabla \psi \|_{L^{2p}(S^2)}^2 \leqslant 
C  \| \nabla \psi \|_{L^2(S^2)} \big(\| \nabla^2 \psi \|_{L^p(S^2)} + \| \nabla \psi \|_{L^2(S^2)}\big). 
\]
Using \eqref{eqnpsi}, \eqref{estmeanpsi}, a Poincar\'e-type inequality, H\"{o}lder's inequality, 
the Gagliardo-Nirenberg estimate from above, 
\eqref{eq:f1bd}, \eqref{eq:f2bd} and Lemma \ref{l:intchi}, we get
\begin{align*}
\|\Delta \psi \|_{L^p(S^2)}\leqslant 
 & C ( \|\psi - \mint_{S^2}\psi \, dA_{S^2} \|_{L^p(S^2)} + 
  |\mint_{S^2}\psi \, dA_{S^2}| + \|\nabla \psi\|_{L^2(S^2)}\\ 
  &+ \|\nabla \psi\|_{L^{2p}(S^2)}^2 + \|f_1\|_{L^p(S^2)} + \|f_2\|_{L^2(S^2)} )\\
\leqslant & C (\delta + \al - 1)(1 + C (\al -1 + \delta) \|\nabla^2 \psi \|_{L^p(S^2)} . 
\end{align*} 
We can insert this estimate into \eqref{est:lp} 
and then choose $\al_0-1$ and $\delta_0$ small in order to get
\[ 
\|\nabla \psi \|_{W^{1,p}(S^2)} \leqslant C (\delta + \al - 1) .
\] 
Using once more \eqref{estmeanpsi} and 
the Sobolev embedding theorem, we get, for any $p \in (\frac43,\frac32]$,
\[ 
\| \psi \|_{L^\infty(S^2)} \leqslant 
C\| \psi-\mint_{S^2} \psi \, dA_{S^2} \|_{W^{2,p}(S^2)} + 
C\left| \mint_{S^2} \psi \, dA_{S^2} \right| \leqslant C (\delta + \al - 1). 
\] 
This concludes the proof. 
\end{proof}

\section{A Bound on $\l$}\label{s:lbound} 
In this section we shall show how the estimates \eqref{eq:sup2p} and \eqref{eq:alphalambd} 
imply a very slow growth on 
$ \frac{\partial}{\partial \log \l} E_{\al,\l} (Id)$ which, when coupled with 
\eqref{growthenergy}, implies a bound on $\l$, independent of how close $\al$ is to 1. 
We start by computing $\frac{d}{d \l}E_{\al,\l}(v)$ directly from \eqref{def:Ealphlam} and \eqref{eq:chil}: 
\begin{align*} 
\log(\chi_{\l}(\z)) &= 2 \log(1+\l^2 |\z|^2) - 2 \log \l - 2 \log(1 + |\z|^2) \\ 
\frac{d}{d \l} \log(\chi_\l(\z)) &= \frac{4\l |\z|^2}{1+\l^2 |\z|^2} - \frac{2}{\l} \\ 
\frac{d}{d \log \l} \log(\chi_\l(\z)) &= \frac{2 (\l^2 |\z|^2- 1)}{\l^2 |\z|^2+1}. 
\end{align*} 
\begin{align*} 
\frac{d}{d \log \l}E_{\al,\l}(v) &= \frac12 \frac{d}{d \log \l} 
\int_{S^2} \left( 2 + \chi_\l |\nabla_{S^2} v|^2 \right)^{\al} \frac{1}{\chi_\l} \, dA_{S^2} \\ 
&= \int_{S^2} ( 2 + \chi_\l |\nabla_{S^2} v|^2)^{\al - 1} 
\left((\al - 1) |\nabla_{S^2} v|^2 - \frac{2}{\chi_{\l}}) \right) z(\l \z) \, dA_{S^2} 
\end{align*} 
where, as in section \ref{s:mob}, $z(\z) := \dfrac{|\z|^2- 1}{|\z|^2+1} \in [-1,1)$. 
 
We wish to estimate \,$\dfrac{d}{d \log \l}E_{\al,\l}(Id) \, - \, \dfrac{d}{d \log \l}E_{\al,\l}(v)$\, 
in terms of a suitable norm of the difference between $Id$ and $v$. 
\begin{align} \label{eq:diff1} 
\frac{d}{d \log \l} E_{\al,\l}&(Id) \, - \, \frac{d}{d \log \l}E_{\al,\l}(v) \notag \\ 
&\hspace{-20pt}= \ - \int_{S^2}  \big( ( 2 + 2 \chi_\l )^{\al - 1} - ( 2 + \chi_\l |\nabla_{S^2} v|^2)^{\al - 1} \big) 
\frac{2 z(\l \z)}{\chi_{\l}}\, dA_{S^2} \\ 
&\hspace{-10pt}+ (\al - 1) \int_{S^2} 
\big( 2\,( 2 + 2 \chi_\l)^{\al - 1} - |\nabla_{S^2} v|^2 \, ( 2 + \chi_\l |\nabla_{S^2} v|^2)^{\al - 1} \big) 
z(\l \z) \, dA_{S^2} \, . \notag 
\end{align} 
As in the proof of Lemma \ref{l:Eaclose}, there is a positive function 
$g : S^2 \to \R_+$ whose value at $p$ lies between 
$|\nabla_{S^2} v(p)|^2$ and 2 = $|\nabla_{S^2} Id|^2$ such that 
\[ 
\big( ( 2 + 2 \chi_\l )^{\al - 1} - ( 2 + \chi_\l |\nabla_{S^2} v|^2)^{\al - 1} \big) 
= (\al - 1)( 2 + g \chi_\l )^{\al - 2} \chi_{\l}(2 - |\nabla_{S^2} v|^2). 
\] 
Similarly, 
\begin{align*} 
2\,( 2 + 2 \chi_\l)^{\al - 1}&\mbox{} - |\nabla_{S^2} v|^2 \, ( 2 + \chi_\l |\nabla_{S^2} v|^2)^{\al - 1} \\ 
&\quad = ( 2 + 2 \chi_\l)^{\al - 1} (2 - |\nabla_{S^2} v|^2) 
\\
&\quad \quad + (\al - 1)( 2 + g \chi_\l )^{\al - 2} \chi_{\l}(2 - |\nabla_{S^2} v|^2) |\nabla_{S^2} v|^2 \,. 
\end{align*} 
If $\al \leqslant 2$, 
\[ 
( 2 + g \chi_\l )^{\al - 2} \leqslant 1. 
\] 
Moreover, 
\begin{align*} 
\frac{\chi_{\l} |\nabla_{S^2} v|^2}{2 + g \chi_\l} &\leqslant 
\begin{cases} 
\frac12 |\nabla_{S^2} v|^2,&\text{if $|\nabla_{S^2} v|^2 \geqslant 2$} \\ 
1,&\text{if $|\nabla_{S^2} v|^2 \leqslant 2$}, 
\end{cases} \\ 
&\leqslant 1 + |\nabla_{S^2} v|^2 
\end{align*} 
and 
\[ 
( 2 + 2 \chi_\l )^{\al - 1} \leqslant 4^{\al - 1} \l^{2 \al - 2}, \qquad 
( 2 + g \chi_\l )^{\al - 1} \leqslant 4^{\al - 1} \l^{2 \al - 2} (1 + |\nabla_{S^2} v|^{2 \al - 2}). 
\] 
Therefore, using that $|z|\leqslant 1$,
\begin{equation}\label{eq:diff2} 
\left |\big( ( 2 + 2 \chi_\l )^{\al - 1} - ( 2 + \chi_\l |\nabla_{S^2} v|^2)^{\al - 1} \big) 
\frac{2 z(\l \z)}{\chi_{\l}} \right| \leqslant 2 (\al - 1) |2 - |\nabla_{S^2} v|^2|
\end{equation} 
and 
\begin{align}\label{eq:diff3} 
\big|\big( 2\,( 2 &+ 2 \chi_\l)^{\al - 1} 
- |\nabla_{S^2} v|^2 \, ( 2 + \chi_\l |\nabla_{S^2} v|^2)^{\al - 1} \big) z(\l \z) \big| \notag \\ 
&\leqslant C \l^{2 \al - 2} \, \big|2 - |\nabla_{S^2} v|^2\big| \, (1 + (\al - 1) |\nabla_{S^2} v|^{2\al}). 
\end{align} 
Using \eqref{eq:diff2} and \eqref{eq:diff3} in \eqref{eq:diff1} we can finally estimate 
\begin{align} \label{eq:diff4} 
\frac{d}{d \log \l}E_{\al,\l}&(Id) \, - \, \frac{d}{d \log \l}E_{\al,\l}(v) \notag \\ 
&\leqslant C(\al - 1)(1 + \l^{2 \al - 2}) 
\int_{S^2}\big|2 - |\nabla_{S^2} v|^2\big| \, (1 + (\al - 1) |\nabla_{S^2} v|^{2\al}) \, dA_{S^2} \notag \\ 
&\leqslant C(\al - 1)(1 + \l^{2 \al - 2})  
\| \nabla(v - Id) \|_{L^2(S^2)} (\,\| \nabla Id \|_{L^2(S^2)} + \| \nabla v \|_{L^2(S^2)}) \\ 
&\quad +C(\al - 1)^2(1 + \l^{2 \al - 2}) \| \nabla(v - Id) \|_{L^{2 \al + 2}(S^2)} \notag \\
&\quad \cdot(\,\| \nabla Id \|_{L^{2 \al + 2}(S^2)} + \| \nabla v \|_{L^{2 \al + 2}(S^2)}) 
\| \nabla v \|^{2 \al}_{L^{2 \al + 2}(S^2)}. \notag  
\end{align} 
\begin{proposition} \label{p:diff5} 
There exist $1<\al_0$, $\delta_0>0$, possibly smaller than those in Proposition \ref{p:clomob2}, 
such that if $v\in W^{1,2\al}(S^2,S^2)$ 
is a critical point of $E_{\al,\l}$ satisfying \eqref{eq:delta-close} and \eqref{eq:alphalambd}, 
$1<\al \leqslant \al_0, \ 0<\delta \leqslant \delta_0$, then 
\begin{equation}\label{eq:lbound1} 
\log \l \leqslant C(\delta + \al - 1). 
\end{equation} 
\end{proposition} 
\begin{proof} 
As in Proposition \ref{p:clomob2}, we set $\psi := v - Id$. By the Sobolev embedding, 
\[ 
\|\nabla \psi\|_{L^{2 \al + 2}(S^2)} \leqslant C(\al) \|\nabla \psi \|_{W^{1,p}(S^2)}, 
\quad p := \frac{2 \al + 2}{\al + 2}. 
\] 
Note that, since we may assume $\al_0 \leqslant 2$, we have that $p \in (\frac43,\frac32]$, 
as in Proposition \ref{p:clomob2}. Moreover, $C(\al)$ can then be chosen independent of $\al$. 
So, taking $\al_0$ and $\delta_0$ as in Proposition \ref{p:clomob2}, we get, from \eqref{eq:sup2p},  
\begin{equation}\label{eq:est1p} 
\|\nabla \psi\|_{L^{2 \al + 2}(S^2)} \leqslant C (\delta + \al - 1).
\end{equation} 
In particular, $\|\nabla v\|_{L^{2 \al + 2}(S^2)} \leqslant 
\|\nabla \psi\|_{L^{2 \al + 2}(S^2)} + \|\nabla Id\|_{L^{2 \al + 2}(S^2)} \leqslant C$. 

By \eqref{eq:alphalambd} we have 
\begin{equation}\label{eq:lpoweral} 
\l^{2 \al - 2} < \max\{e^{2C \delta}, e^{2 \al_0 - 2}\}. 
\end{equation} 

Since $v$ is a critical point of $E_{\al,\l}$ we have 
$\left.\frac{d}{d \log \tau}\right|_{\tau = \l}E_{\al,\tau}(v) = 0$. In order to see this we note that
\[
E_{\al,\tau}(v) = E_{\al,\l}(v_{\l \tau^{-1}}) 
\]
which gives 
\[
\frac{d}{d \log \tau} E_{\al,\tau}(v) |_{\tau=\l} = 
\left( \tau \frac{d}{d \tau}E_{\al,\tau}(v) \right) |_{\tau=\l}= E'_{\al,\l}(v)(w), 
\]
where $w$ is the vector field along $v$ given by 
\[
w = \left( \tau \frac{d}{d \tau} v_{\l \tau^{-1}}\right) |_{\tau=\l}.
\] 
But $v$ is a critical point of $E_{\al,\l}$ and therefore $E'_{\al, \l}(v) = 0$.
 
It then follows from \eqref{growthenergy}, \eqref{eq:diff4}, \eqref{eq:est1p} and \eqref{eq:lpoweral} that 
\begin{equation}\label{eq:lbound2} 
C'^{\,-1} (\al - 1) \frac{\log \l}{1+ \log \l} \leqslant \frac{d}{d \log \l}E_{\al,\l}(Id) \leqslant 
C(\al - 1)(\delta + \al - 1). 
\end{equation} 
The estimate \eqref{eq:lbound1} now follows by taking $\al_0 - 1$ and $\delta_0$ sufficiently small. 
\end{proof} 

\section{Optimal $\l$ and Better Closeness in the $W^{2,p}$-norm} \label{s:betterclosw2p} 
Of course, we wish to prove that $\l = 1$. However, the choice of $\l$ 
provided by Proposition \ref{p:clomob} has some flexibility and therefore, 
at the moment, we cannot hope to do better than \eqref{eq:lbound1}. 
So we have to choose $\l$ optimally, which we do as follows. 

Proposition \ref{p:clomob} suggests that we should choose $M$ so as to minimize 
$\| \nabla (u_M - Id) \|^2_{L^2(S^2)} = \| \nabla (u - M^{-1}) \|^2_{L^2(S^2)} $. 
This minimization is possible because, as $M \to \infty$ in the M\"{o}bius group $PSL(2,\C)$, 
$\| \nabla (u - M^{-1}) \|^2_{L^2(S^2)} \to \| \nabla u \|^2_{L^2(S^2)} + \| \nabla Id \|^2_{L^2(S^2)} 
\geqslant 16 \pi$ and therefore, we only need to minimize $\| \nabla (u_M - Id) \|^2_{L^2(S^2)} $ 
over a compact subset of $PSL(2,\C).$ 
In order to see this we note that up to rotations, $M$ can only go to infinity if 
it approaches a dilation from the south pole towards the north pole by a huge factor $\l$, 
so that the energy of $m_{\l}$ is concentrated on a small disk $D$ centred at the south pole. 
Take $D$ so small that the energy of $u$ on $D$ is less than $\ep$ and 
the energy of $m_{\l}$ outside of $D$ is less than $\ep$. By breaking up the integral for 
\[
\| \nabla (u-M^{-1}) \|_{L^2(S^2)}^2 = \| \nabla u \|_{L^2(S^2)}^2 + 2 \langle \nabla u, \nabla M^{-1} \rangle_{L^2(S^2)} + \| \nabla M^{-1} \|_{L^2(S^2)}^2
\]
into the contributions from $D$ and its complement, we see that  
\[
\langle \nabla u, \nabla M^{-1} \rangle_{L^2(S^2)}
\]
is small and noting that by conformal invariance 
$\| \nabla M^{-1}\|_{L^2(S^2)}=\| \nabla Id\|_{L^2(S^2)}$, the claim follows. 

From now on, 
we shall assume that $M$ does minimize $\| \nabla (u_M - Id) \|_{L^2(S^2)}$. 
Of course, all the estimates proved so far still hold. 

As usual, we set $v := u_M$ and assume that 
$v$ satisfies the hypotheses of Proposition \ref{p:diff5}. 
We notice that, by \eqref{eq:sup2p}, 
$v$ approaches the identity map pointwise as $\delta$ and $(\al-1)$ tend to zero. 
So we may write 
\[ 
v = Id + \psi = \exp_{Id} \hat{\psi} \quad (= Id + \hat{\psi}  + O(|\hat{\psi}|^2)); 
\qquad \hat{\psi} \in T_{Id} W^{1,2\al}(S^2,S^2).  
\] 
More explicitly, if $\mathbf{x} = (x,y,z) \in S^2 \subset \R^3$, then 
\begin{gather} 
v(\mathbf{x}) = \mathbf{x} \sqrt{1 - |\hat{\psi}(\mathbf{x})|^2} 
\, + \, \hat{\psi}(\mathbf{x}), \qquad \hat{\psi}(\mathbf{x}) \cdot \mathbf{x} \equiv 0. \notag \\ 
\hat{\psi}(\mathbf{x}) = \psi(\mathbf{x}) + \tfrac12 |\psi(\mathbf{x})|^2 \mathbf{x}\,, \qquad 
\psi(\mathbf{x}) = \hat{\psi}(\mathbf{x}) -
\left( 1 - \sqrt{1 - |\hat{\psi}(\mathbf{x})|^2}\right) \mathbf{x},  \label{eq:psipsih} \\ 
|\hat{\psi}|^2 = |\psi|^2(1 - \tfrac14|\psi|^2) \leqslant |\psi|^2 = 2(1 - \sqrt{1 - |\hat{\psi}|^2}). \notag 
\end{gather} 
It follows that 
\begin{align}
|\nabla \psi -\nabla \hat{\psi}|&=O(|\hat{\psi}| \,  |\nabla \hat{\psi}|) + O(|\hat{\psi}|^2)
= O(|\psi| \,  |\nabla \psi|) + O(|\psi|^2),\nonumber \\
|\nabla^2 \psi -\nabla^2 \hat{\psi}| &=O(|\hat{\psi}| |\nabla^2 \hat{\psi}|) 
+ O(|\nabla \hat{\psi}|^2) + O(|\hat{\psi}|^2)= O(|\psi| |\nabla^2 \psi|) 
+ O(|\nabla \psi|^2) +O(|\psi|^2)\label{eq:page14}
\end{align} 
and therefore, we derive the following equation for $\hat{\psi}$ by taking 
the component of \eqref{eqnpsi} orthogonal to the identity: 
\begin{align}
(\Delta \hat{\psi})^T + 2 \hat{\psi} = & - 2\langle\nabla \hat{\psi} ,\nabla Id \rangle \hat{\psi} 
- f_1^T - f_2^T + O(| \nabla \hat{\psi}|^2)+O(|\hat{\psi}|^2) \label{eqnpsih}
\end{align} 
where $T$ denotes orthogonal projection of a vector at $\mathbf{x} \in S^2$ onto 
$T_{\mathbf{x}}S^2$, i.e. onto the orthogonal complement of $\mathbf{x}$, 
and $f_1$ and $f_2$ are given by \eqref{def.fi} and \eqref{def.fi1}. 

Next, we let $e_1$, $e_2$ be an orthonormal basis for $T_{\mathbf{x}}S^2$ so that 
$D_{e_i} e_j(\mathbf{x}) = 0$, where $D$ is the covariant derivative on $TS^2$. 
We calculate at $\mathbf{x}$: 
\[
D_{e_i} \hat{\psi}(\mathbf{x}) = 
e_i(\hat{\psi})(\mathbf{x}) - ((e_i(\hat{\psi}) \cdot \mathbf{x})\mathbf{x} = 
e_i(\hat{\psi})(\mathbf{x})+(\hat{\psi}(\mathbf{x})\cdot e_i(\mathbf{x})) \mathbf{x}
\]
and, since $\hat{\psi}(\mathbf{x}) = \sum_{i=1}^2 (\hat{\psi}(\mathbf{x}) \cdot e_i)e_i$, 
we conclude that 
\[ 
(\Delta \hat{\psi})^T + \hat{\psi} = \Delta_{TS^2}\hat\psi 
\] 
where $\Delta_{TS^2}$ is the (rough) connection Laplacian on vector fields on $S^2$. 
Next it follows from \cite{chavel}, Proposition A3, that
\[
\Delta_H \hat{\psi} =\Delta_{TS^2}\hat\psi -\hat\psi,
\]
where $\Delta_H$ is the (negative semi-definite) Hodge Laplacian.
Furthermore, it was calculated in \cite{smith75} that
\[ 
-\Delta_{TS^2}\hat{\psi} - \hat{\psi} =  -(\Delta \hat{\psi})^T -2\hat{\psi}= J \hat{\psi} 
\] 
where $J$ is the Jacobi operator of the energy functional at the identity on $S^2$. 
By standard Hodge theory, the spectrum of $\Delta_{TS^2}$ is the same as 
the spectrum of $\Delta$ on functions shifted up by 1, i.e., 
the spectrum of $\Delta_{TS^2}$ is $\{-1, -5, \dotsc\}$. 
Indeed, if $\Delta \phi + c \phi = 0$ then 
$\Delta_{TS^2} (\nabla \phi) + (c-1) \nabla \phi = 0$ and 
$\Delta_{TS^2} (*\nabla \phi) + (c-1) (*\nabla \phi) = 0$ 
where $*$ is rotation by $90^{\circ}$ in $TS^2$. 
These two equations follow from the above relation between 
$\Delta_H$ and $\Delta_{TS^2}$ and the facts that 
the exterior derivative $d$ and $*$ both commute with $\Delta_H$; 
the second equation follows from the first and 
the conformal invariance of the Dirichlet integral in two dimensions. 
So, the kernel of $J$ consists precisely of the span of 
the gradient of the linear functions on $S^2$ and their $90^{\circ}$ rotations. 
But this is precisely the tangent space $Z$ of the M\"{o}bius group at the identity; 
the flow of the gradient of a linear function is a dilation and 
the flow of a $90^{\circ}$ rotation of the gradient of a linear function is a rotation. 

We shall be making use of the elliptic estimate 
\[ 
\|\hat\psi\|_{W^{2,p}} \leqslant C (\| J\hat{\psi} \|_{L^p} + \| \hat{\psi}_0 \|_{L^p}) 
\] 
where $\hat{\psi}_0$ is the orthogonal projection of $\hat{\psi}$ onto the kernel of $J$ 
with respect to the inner product on $L^2(S^2)$. 
We start by estimating $\hat{\psi}_0$. 
From the minimizing property of $\| \nabla (v - Id) \|^2_{L^2(S^2)}$ it follows that 
\[ 
-\int_{S^2}\nabla v \cdot \nabla \xi \, dA_{S^2} + \int_{S^2}\nabla Id \cdot \nabla \xi \, dA_{S^2} = 0 
\quad \forall \, \xi \in Z.  
\] 
Now $\nabla Id \cdot \nabla \xi = \Div \xi$ and $\int_{S^2} (\Div \xi) \, dA_{S^2} = 0$. Therefore 
\begin{equation}\label{eq:orth1} 
\int_{S^2} v \cdot \Delta \xi \, dA_{S^2} = 0 \quad \forall \, \xi \in Z. 
\end{equation} 
We have 
\[ 
\Delta \xi (\mathbf{x}) = (\Delta \xi)^T(\mathbf{x}) + (\Delta \xi \cdot \mathbf{x})\mathbf{x} 
\] 
and, since $\xi \in Z, \ (\Delta \xi)^T = - 2 \xi$. If, as before, 
$e_1$, $e_2$ is an orthonormal basis for $T_{\mathbf{x}}S^2$ so that 
$D_{e_i} e_j(\mathbf{x}) = 0$, then 
\begin{align*} 
\Delta \xi \cdot \mathbf{x} &= \sum_{i=1}^2 \bigg(e_i\big(e_i(\xi) \cdot \mathbf{x} \big) - 
\big( e_i(\xi) \cdot e_i \big)(\mathbf{x})\bigg) \\ 
&= -\sum_{i=1}^2 \bigg(e_i(\xi \cdot e_i)(\mathbf{x}) + 
\big( e_i(\xi) \cdot e_i \big)(\mathbf{x})\bigg)\\ 
&= -\sum_{i=1}^2 \bigg(\big(e_i\big(\xi) \cdot e_i \big)(\mathbf{x}) + 
\big( e_i(\xi) \cdot e_i \big)(\mathbf{x})\bigg)\\ 
& = - 2 \Div \xi(\mathbf{x}) 
\end{align*} 
where we used $\xi \cdot \mathbf{x} = 0$ in the second line and 
$\xi \cdot e_i(e_i) = \xi \cdot D_{e_i}e_i = 0$ in the third line.
Using these calculations of $\Delta \xi$ in \eqref{eq:orth1} yields 
\[ 
\int_{S^2} v \cdot \xi \, dA_{S^2} + \int_{S^2} (v \cdot \mathbf{x}) (\Div \xi) \, dA_{S^2} = 0, 
\] 
and, taking into account \eqref{eq:psipsih}, the fact that $\xi$ is tangent to $S^2$ and 
$\int_{S^2} (\Div \xi) \, dA_{S^2} = 0$, we obtain 
\[ 
\int_{S^2} \hat{\psi} \cdot \xi \, dA_{S^2} = 
- \int_{S^2} \sqrt{1 - |\hat{\psi}|^2} (\Div \xi) \, dA_{S^2} = 
\int_{S^2} \big(1 - \sqrt{1 - |\hat{\psi}|^2}\big) (\Div \xi) \, dA_{S^2}. 
\] 
We now choose $\xi = \hat{\psi}_0$ and get 
\[ 
\|\hat{\psi}_0\|^2_{L^2(S^2)} \leqslant 
\|\hat{\psi}\|^2_{L^{\infty}(S^2)} \int_{S^2} |\nabla \hat{\psi}_0| \, dA_{S^2}. 
\] 
But $(\Delta \hat{\psi}_0)^T = -2 \hat{\psi}_0$ because $\hat{\psi}_0 \in Z$ and therefore 
\begin{align*}
\int_{S^2} |\nabla \hat{\psi}_0| \, dA_{S^2} \leqslant& 
C \left( \int_{S^2} |\nabla \hat{\psi}_0|^2 \, dA_{S^2}\right)^{1/2} = \ 
2C \left( \int_{S^2} -\Delta \hat{\psi}_0 \cdot \hat{\psi}_0 \, dA_{S^2} \right)^{1/2}\\
=&  2C \|\hat{\psi}_0\|_{L^2(S^2)}. 
\end{align*}
We have proved that, for $p \in [\frac43,\frac32]$, 
\begin{equation} \label{eq:orth2} 
\|\hat{\psi}_0\|_{L^p(S^2)} \leqslant C \|\hat{\psi}_0\|_{L^2(S^2)} \leqslant 
C \|\hat{\psi}\|^2_{L^{\infty}(S^2)} \leqslant 
C \|\hat{\psi}\|_{L^{\infty}(S^2)} \|\hat\psi\|_{W^{2,p}} . 
\end{equation} 

We next estimate $\| J\hat{\psi} \|_{L^p}$ by 
estimating the $L^p$ norm of the right hand side of \eqref{eqnpsih}. 

From \eqref{eq:f2bd}, \eqref{eq:gradlogchiexplicit} and \eqref{eq:lbound1} we have, 
\[ 
|f_2| \leqslant C(\al - 1) \, (\sup |\nabla \log \chi_{\l}|) \, |\nabla v| \leqslant C(\al - 1) (\log \l) |\nabla v|, 
\] 
where we have used $(\l - 1) \leqslant C (\log \l)$ which holds 
because of the bound \eqref{eq:lbound1} on $\l$. 
Therefore, 
\begin{equation} \label{eq:f2bd2} 
\|f_2^T\|_{L^p(S^2)} \leqslant C(\al - 1) (\log \l) \|\nabla v\|_{L^p(S^2)}. 
\end{equation} 

To estimate $\|f_1\|_{L^p(S^2)}$ we recall that 
\[ 
|\nabla v|^2 = |\nabla Id|^2 + 2 \langle \nabla Id , \nabla \psi \rangle + |\nabla \psi |^2 
= 2 + 2\Div \psi +|\nabla \psi|^2
\] 
and therefore, 
\[ 
\big| \nabla (|\nabla v|^2) \big| \leqslant C \, |\nabla^2 \psi| \, (1 + |\nabla v|).
\] 
It follows from \eqref{def.fi} and the estimate
\begin{align*}
\frac{\chi_\l|\nabla v|(1+|\nabla v|)}{2+\chi_\l |\nabla v|^2} 
\leqslant \frac12 \sqrt{\chi_\l} +1 \leqslant 1 + \l \leqslant C
\end{align*}
that 
\begin{equation}\label{eq:f1bd2} 
|f_1| \leqslant C (\al - 1) |\nabla^2 \psi| \left( 
\frac{\chi_{\l} | \nabla v| (1 + |\nabla v|)}{2 + \chi_{\l} |\nabla v|^2} \right) 
\leqslant C  (\al - 1) |\nabla^2 \psi| 
\end{equation} 
where we have used $\chi_{\l} < \l^2$ and the bound \eqref{eq:lbound1} on $\l$. 

Using these bounds on $f_1$ and $f_2$ and \eqref{eq:page14} in \eqref{eqnpsih}, 
keeping in mind that $\|\nabla v\|_{L^p(S^2)}$ is bounded by the energy of $v$, 
we see, also using \eqref{eq:orth2}, that 
\begin{align*} 
\|\hat\psi\|_{W^{2,p}} & \leqslant C (\| J\hat{\psi} \|_{L^p} + \| \hat{\psi}_0 \|_{L^p}) \\ 
&\leqslant C \|\hat{\psi}\|_{L^{\infty}(S^2)} \|\nabla \hat{\psi}\|_{L^p(S^2)} 
+ C(\al - 1) \big(\| \nabla^2 \hat{\psi} \|_{L^p(S^2)} + (\log \l) \big) \\ 
& \hphantom{4 \|\hat{\psi}\|_{L^{\infty}(S^2)} \|\nabla } 
+ C \|\nabla \hat{\psi}\|_{L^{2p}(S^2)}^2 + C \|\hat{\psi}\|_{L^{\infty}(S^2)} \|\hat\psi\|_{W^{2,p}} . 
\end{align*} 
We now appeal to the Gagliardo-Nirenberg interpolation inequality 
\[ 
\| \nabla \hat{\psi} \|_{L^{2p}(S^2)}^2 \leqslant 
C  \| \nabla \hat{\psi} \|_{L^2(S^2)} \| \nabla \hat{\psi} \|_{W^{1,p}(S^2)} 
\] 
and use \eqref{eq:sup2p} with $\delta_0$ and $\al_0 - 1$ sufficiently small, to conclude that 
\begin{equation} \label{eq:betterest} 
\|\hat\psi\|_{W^{2,p}} \leqslant C(\al - 1)(\log \l). 
\end{equation} 

\section{Proof of Theorem \ref{t:main}} \label{s:pf}
We start with a classification result for $\al$-harmonic maps of degree $0$ with ``small" energy.
\begin{proposition} \label{p:deg0} 
Fix $\eta > 0$. Then there exists $\overline{\al}-1 > 0$  small, 
$\overline{\al}$ depending only on $\eta$, 
such that if $1 < \al \leqslant \overline{\al}$ and $u \colon S^2 \to S^2$ is $\al$-harmonic, 
of degree zero and $E(u) \leqslant 8 \pi - \eta$, then $u$ is constant. 
\end{proposition} 

\begin{proof} 
If the proposition is not true, then we can find a sequence $\al_j \searrow 1$ 
and a sequence of non-constant maps $u_j \colon S^2 \to S^2$ such that 
$\deg(u_j) = 0, \ u_j$ is $\al_j$-harmonic and $E(u_j) \leqslant 8 \pi - \eta \ \forall \, j \in \N$. 
By the results of Sacks-Uhlenbeck \cite{sacks81} we know that two possibilities can occur: 
\begin{enumerate} 
\item[(i)] $u_j$ converges smoothly to a harmonic map $u^* \colon S^2 \to S^2$ of degree zero which is therefore constant, or
\item[(ii)] there exist two harmonic maps $u^* \colon S^2 \to S^2$ 
and $u^B \colon S^2 \to S^2$ and a point $p \in S^2$ such that, 
a subsequence of $u_j$ (still denoted by $u_j$) converges smoothly 
on compact subsets of $S^2 \setminus \{p\}$ to $u^*$ and 
a nontrivial bubble $u^B$ develops at $p$. Since $E(u^B) < 8 \pi$ 
we have $|\deg(u^B)| = 1$. By choosing the orientation of 
the domain $S^2$ relative to that of the image $S^2$ appropriately, 
we may, and we will, assume that $4 \pi \deg(u^B) = E(u^B) = 4 \pi$. 
(It follows that $u^*$ is constant, but this is not of direct importance to us.) 
\end{enumerate} 
In case (i), $E(u_j) \to 0$ as $j \to \infty$. But then, by Theorem 3.3 
in Sacks-Uhlenbeck \cite{sacks81}, 
there exists $\ep > 0$ and $\al_0 > 1$ such that, 
if $v$ is $\al$-harmonic, $1 \leqslant \al < \al_0$ and $E(v) < \ep$ 
then $v$ is constant. In particular, $u_j$ is constant for large enough $j$, 
contrary to our assumption. 

In case (ii), we can find a sequence $D_j$ of discs centred at $p$, 
whose radii $r_j$ decrease to $0$ and a sequence $\sigma_j \searrow 0$ 
such that $\sigma_j/ r_j \uparrow +\infty$ and, if 
\[ 
v_j(z) := u_j(r_jz), \quad |z| < \sigma_j /r_j, 
\] 
then 
\[ 
\sup_{|z| < \sigma_j/ r_j} (|v_j(z) - u^B(z)| + |\nabla v_j(z) - \nabla u^B(z)|) \to 0 
\ \ \text{as} \ \ j \to \infty. 
\] 
In particular, 
\[ 
\int_{D_j} J(u_j) \, dA_{S^2} \to 4 \pi \deg(u^B) = 4 \pi \ \ \text{as} \ \ j \to \infty 
\] 
and 
\[ 
\int_{D_j} |\nabla u_j|^2 \, dA_{S^2} \to \int_{S^2} |\nabla u^B|^2 \, dA_{S^2} = 8 \pi 
\ \ \text{as} \ \ j \to \infty. 
\] 

But then, for large enough $j$, 
\begin{align*} 
\int_{S^2} J(u_j) \, dA_{S^2} &= \int_{D_j} J(u_j) \, dA_{S^2} + 
\int_{S^2 \setminus D_j} J(u_j) \, dA_{S^2} \\ 
&\geqslant (4 \pi - \tfrac14 \eta) - 
\frac12 \int_{S^2 \setminus D_j} |\nabla u_j|^2 \, dA_{S^2} \\ 
&\geqslant (4 \pi - \tfrac14 \eta) - \big((8 \pi - \eta) - (4 \pi - \tfrac14 \eta)\big) \\ 
& = \tfrac12 \eta > 0. 
\end{align*} 
Therefore, for large enough $j, \ u_j$ has nonzero degree, which is again contrary to our assumption. 
\end{proof}

\begin{proof}[Proof of Theorem \ref{t:main}]
Since we have Proposition \ref{p:deg0} at our disposal, 
we only need to classify the $\al$-harmonic maps of degree $1$ 
which satisfy the assumptions of Theorem \ref{t:main}. 

In order to do this, we go back to the proof of Proposition \ref{p:diff5}, 
using our improved estimate \eqref{eq:betterest} to obtain 
\[ 
\|\nabla \psi\|_{L^{2 \al + 2}(S^2)} \leqslant C (\al - 1) (\log \l). 
\] 
The string of inequalities in \eqref{eq:lbound2} now becomes 
\[ 
C'^{\,-1} (\al - 1) \frac{\log \l}{1+ \log \l} \leqslant \frac{d}{d \log \l}E_{\al,\l}(Id) \leqslant 
C(\al - 1)^2(\log \l). 
\] 
By demanding that $\al$ be suffciently close, but not equal, to 1, we conclude that $\l = 1$. 
But, by \eqref{eq:betterest} this implies that $\hat{\psi}$ must vanish, that is, 
$v$ is the identity and the M\"{o}bius transformation $M$ 
which minimizes $\| \nabla (u_M - Id) \|^2_{L^2(S^2)}$ must be a rotation. 
So $u$ is a rotation, as claimed. 
\end{proof} 

\section{Other $\al$-harmonic maps of degree 1} \label{s:other} 
In this section we shall construct rotationally symmetric $\al$-harmonic maps of degree 1 
that are not rotations. Of course, their $\al$-energy will be strictly bigger than $2^{2 \al + 1} \pi$.
We shall also construct $\al$-harmonic maps of degree 1 from the disk to the sphere 
which map the boundary circle to a point. 
This was proved to not be possible for a harmonic map by Lemaire (see, for instance, (12.6) in \cite{eelm}). 
We shall further construct a map of degree 1 from the annulus to the sphere 
which is $\al$-harmonic and which maps the boundary circles to antipodal points. 

\subsection{Rotationally symmetric maps} 
For $n \in \N, \ r \in [n\pi, (n+1)\pi]$ and $\theta \in [0,2\pi]$, 
we consider a parameterisation of $S^2$ given by 
\[ 
(r,\theta) \mapsto (\sin r \, \cos \theta, \ \sin r \, \sin \theta, \ \cos r). 
\] 
This parameterisation is orientation preserving if $n$ is even and orientation reversing if $n$ is odd. 
In these coordinates, the metric on $S^2$ is given by 
\[ 
dr^2 + (\sin r)^2 d \theta^2. 
\] 

We shall be interested in maps $u_f$ from $S^2$ to itself which are of the form 
\[ 
(r,\theta) \mapsto (\sin (f(r)) \cos \theta, \ \sin (f(r)) \sin \theta, \ \cos (f(r)))
\]
with
\[
f \colon [0,\pi] \to \R, \ f(0) = 0, \ f(\pi) = n \pi. 
\] 
These maps are rotationally symmetric and, for $n > 1$, wrap over $S^2$ more than once; 
the degree is zero if $n$ is even and one if $n$ is odd. 
The energy density $e(u_f)$ of such a map is given by 
\[ 
e(u_f) = \frac12 \left((f')^2 + \frac{(\sin f)^2}{(\sin r)^2}\right) 
\] 
and, in order to express the $\al$-harmonic map operator \eqref{EL3} at $u_f$, we compute: 
\begin{gather*} 
\frac{\partial u_f}{\partial r} = f'(r)\big(\cos (f(r)) \cos \theta, \ \cos (f(r)) \sin \theta, \ -\sin (f(r))\big), \\ 
\frac{\partial u_f}{\partial \theta} = \big(- \sin (f(r)) \sin \theta, \ \sin (f(r)) \cos \theta, \ 0\big), \\ 
\frac{\partial^2 u_f}{\partial r^2} = \frac{f''(r)}{f'(r)} \frac{\partial u_f}{\partial r} - (f'(r))^2 u_f, \\ 
\frac{\partial^2 u_f}{\partial \theta^2} = - \sin (f(r)) (\cos \theta, \ \sin \theta, \ 0) = 
- \sin (f(r)) \left(\sin(f(r)) u_f + \frac{\cos(f(r))}{f'(r)} \frac{\partial u_f}{\partial r}\right) .
\end{gather*} 
The Laplacian writes as $\Delta = \frac{\partial^2}{\partial r^2} + \frac{\cos r}{\sin r} \frac{\partial}{\partial r} + 
\frac{1}{(\sin r)^2} \frac{\partial^2}{\partial \theta^2}$ and so, 

\begin{align*} 
\Delta u_f &+ |\nabla u_f|^2 u_f + 
(\al-1) (2+|\nabla u_f|^2)^{-1} \nabla (|\nabla u_f|^2) \cdot \nabla u_f \\[2\jot] 
&\qquad = \frac{f''(r)}{f'(r)} \frac{\partial u_f}{\partial r} - (f'(r))^2 u_f + 
\frac{\cos r}{\sin r} \frac{\partial u_f}{\partial r} \\
&\qquad\quad- \frac{\sin (f(r))}{(\sin r)^2} \left(\sin(f(r)) u_f + 
\frac{\cos(f(r))}{f'(r)} \frac{\partial u_f}{\partial r}\right) \\[2\jot] 
&\qquad\quad + \left((f')^2 + \frac{(\sin f)^2}{(\sin r)^2}\right)u_f + 
\frac{(\al - 1)}{(2+|\nabla u_f|^2)} \frac{\partial |\nabla u_f|^2}{\partial r} \frac{\partial u_f}{\partial r} \\[2\jot] 
&\qquad = \frac{1}{f'(r)} \frac{\partial u_f}{\partial r}(f''(r) + \frac{\cos r}{\sin r} f'(r) - 
\frac{(\cos f(r)) (\sin f(r))}{(\sin r)^2}\\
&\qquad\quad  + \frac{(\al - 1)}{(2+|\nabla u_f|^2)} \frac{\partial |\nabla u_f|^2}{\partial r} ). 
\end{align*} 
Thus $u_f$ is $\al$-harmonic if 
\begin{equation}\label{eq:uf} 
f''(r) + \frac{\cos r}{\sin r} f'(r) - \frac{(\cos f(r)) (\sin f(r))}{(\sin r)^2} + 
\frac{(\al - 1)}{(2+|\nabla u_f|^2)} \frac{\partial |\nabla u_f|^2}{\partial r} = 0. 
\end{equation} 

\subsection{Construction of rotationally symmetric $\al$-harmonic maps} 
We shall specialise to the case $n=3$ 
(though our arguments will work for any other integer value of $n$) and we define 
\[ 
X := \{ f \colon [0,\pi] \to \R : u_f \in W^{1,2\al}(S^2, \R^3), \ \ f(0) = 0, \ f(\pi) = 3 \pi\}. 
\] 
Let $\Lambda := \inf_{f\in X}I(f)$ where 
\[ 
I(f) := \Ea(u_f) = \pi \int_0^{\pi} \left(2 + (f')^2 + \frac{(\sin f)^2}{(\sin r)^2}\right)^{\al} \sin r \, dr . 
\] 
A direct calculation shows that $f \in X$ is a critical point of $I$ if, and only if, 
$u_f$ is an $\al$-harmonic map, i.e., if, and only if, $f$ satisfies \eqref{eq:uf}. 
This is a manifestation of the principle of symmetric criticality of Palais; see, for example, 
Remark 11.4(a) in \cite{amma}. The symmetry group in question here is the group $O(2)$ 
of the rotations about the axis $(0,0,z)$ and reflections in planes containing the line $(0,0,z)$. 

If $f_j$ is a sequence in $X$, we shall write $u_j$ instead of $u_{f_j}$. 
Let $f_j$ be a sequence in $X$ such that $I(f_j) \downarrow \Lambda$. Then 
$u_j$ is a bounded sequence in $W^{1,2\al}(S^2, \R^3)$ and therefore, 
a subsequence, still denoted by $u_j$, converges weakly in $W^{1,2\al}(S^2, \R^3)$ 
and uniformly in $C^0(S^2, \R^3)$ to $u^* := u_{f^*}$ for some $f^* \in X$.\footnote{This 
uniform convergence in $C^0$ fails when $\al =1$ and this is precisely why this construction 
does not yield harmonic maps of the type considered in this section.} 
By the lower semi-continuity of $\Ea$ with respect to weak convergence in $W^{1,2\al}(S^2, \R^3)$, 
we have that $I(f^*) = \Ea(u^*) = \Lambda$. Thus $u^*$ is an $\al$-harmonic map of degree 1 
which is not a rotation. We get a lower bound on $\Ea(u^*)$ by arguing as in \eqref{est1} and \eqref{est2}: 
\begin{align*} 
\Ea(u^*) &= \pi \int_0^{\pi} \left(2 + (f^*\mbox{}')^2 + \frac{(\sin f^*)^2}{(\sin r)^2}\right)^{\al} \sin r \, dr \\[2\jot] 
&\geqslant \pi \left(\int_0^{\pi} \left(2 + (f^*\mbox{}')^2 + \frac{(\sin f^*)^2}{(\sin r)^2}\right) \sin r \, dr \right)^{\al} 
 \left(\int_0^{\pi} \sin r \, dr \right)^{1 - \al} \\[2\jot] 
&\geqslant 2^{1 - \al} \pi \left(\int_0^{\pi} \left(2 \sin r + 2 |f^*\mbox{}'(\sin f^*)| \right) \, dr \right)^{\al}. 
\end{align*} 
There exist $r_1, \, r_2 \in (0, \pi)$ such that $f^*(r_1) = \pi$ and $f^*(r_2) = 2 \pi$. 
Then 
\begin{align*} 
\int_0^{\pi} |f^*\mbox{}'(\sin f^*)| \, dr &\geqslant \int_0^{r_1} f^*\mbox{}'(\sin f^*) \, dr - 
\int_{r_1}^{r_2} f^*\mbox{}' (\sin f^*) \, dr + \int_{r_2}^{\pi} f^*\mbox{}' (\sin f^*) \, dr \\ 
&= \left.-\cos f^*(r)\right|_0^{r_1} + \left.\cos f^*(r)\right|_{r_1}^{r_2} - \left.\cos f^*(r)\right|_{r_2}^{\pi} \\ 
&= 6. 
\end{align*} 
It follows that 
\[ 
\Ea(u^*) \geqslant  2^{3 \al + 1} \pi. 
\] 

Let $D_1$ be the geodesic disc in $S^2$ of radius $r_1$ and centred at $(0,0,1)$, 
let $D_2$ be the geodesic disc in $S^2$ of radius $r_2$ and centred at $(0,0,-1)$ 
and let $A$ be the annulus between $D_1$ and $D_2$. 
Then the restriction of $u^*$ to $D_1$ is an $\al$-harmonic map of degree 1 onto all of $S^2$ 
which maps all of the boundary of $D_1$ to $(0,0,-1)$. 
Similarly, the restriction of $u^*$ to $A$ is an $\al$-harmonic map of degree 1 onto all of $S^2$ 
which maps the two boundaries of $A$ to antipodal points of $S^2$. 

\appendix

\section{An Estimate for the function $\chi_\l$}
\begin{lemma}\label{l:intchi}
There is a constant $C>0$, independent of $\l \geqslant 1$, such that 
\begin{equation}\label{eq:gradlogchi}
  \| \nabla \log \chi_\l \|_{L^2(S^2)} \leqslant 
  \begin{cases}
  C (\log \l)  & \hbox{ for } 0 \leqslant \log \l \leqslant 1; \\ 
  C (\log \l)^{\frac{1}{2}}  & \hbox{ for }  \log \l \geqslant 1.
  \end{cases}
\end{equation} 
\end{lemma}
\begin{proof}
First of all we note that
\begin{equation}\label{eq:gradlogchiexplicit}
\frac{d}{dr} \log \chi_\l(r) 
 = 4 \left( \frac{\l^2 r}{1+\l^2r^2} - \frac{r}{1+r^2} \right) 
 = \frac{4r(\l^2-1)}{(1+r^2)(1+\l^2r^2)}, 
\end{equation}
and hence we estimate
\begin{align*} 
\| \nabla \log \chi_\l \|_{L^2(S^2)} = 
 & 4(\l^2-1) \left( 8 \pi 
   \int_0^\infty \frac{r^3}{(1+\l^2r^2)^2(1+r^2)^4} \, dr 
   \right)^{1/2} \\
 \leqslant & 4(\l^2-1) (8 \pi)^{1/2} \\ & \qquad \left( 
   \int_0^{1/\l} r^3 dr + 
   \frac{1}{\l^4} \int_{1/\l}^1 \frac{1}{r} \, dr + 
   \frac{1}{\l^4} \int_1^\infty \frac{1}{r^9} \, dr 
   \right)^{1/2}. 
\end{align*} 
So, 
\[ 
\| \nabla \log \chi_\l \|_{L^2(S^2)} \leqslant 
4(8 \pi)^{1/2}  \left(\frac{\l+1}{\l}\right) 
 \left(\frac{\l - 1}{\l}\right) 
 \left(\frac{1}{4} + \frac{1}{8} + \log \l\right)^{1/2}. 
\] 
Now, for $1 \leqslant \l \leqslant e$, we have 
\[ 
\frac{\l - 1}{\l} \leqslant \log \l 
\quad\text{and}\quad
(\tfrac14 + \tfrac 18 + \log \l)^{1/2} \leqslant \sqrt{2} 
\] 
and, for $\log \l \geqslant 1$, we have 
\[ 
\frac{\l - 1}{\l} (\tfrac14 + \tfrac 18 + \log \l)^{1/2} \leqslant 
\sqrt{2} (\log \l)^{1/2} 
\] 
which yield the desired estimate \eqref{eq:gradlogchi}. 
\end{proof}

\end{document}